\documentclass{amsart}
\usepackage{amsmath, amssymb}

\newtheorem{theorem}{Theorem}[section]
\newtheorem{lemma}[theorem]{Lemma}
\newtheorem{corollary}[theorem]{Corollary}

\newtheorem{proposition}[theorem]{Proposition}
\newtheorem{claim}[theorem]{Claim}

\theoremstyle{definition}

\newtheorem{remark}[theorem]{Remark}
\newtheorem{definition}[theorem]{Definition}

\def\cD{\mathcal{D}}
\def \DD {\mathcal D}
\def \cK {\mathcal K}

\def \T {\Theta}
\def \t {\theta}
\def \d {\delta}
\def \L {\Lambda}
\def \dd {\partial}
\def \D {\Delta}

\def\dcl{\operatorname{dcl}}
\def\acl{\operatorname{acl}}
\def\alg{\operatorname{alg}}
\def\tp{\operatorname{tp}}
\def\id{\operatorname{id}}
\def\res{\operatorname{res}}
\def\pr{\operatorname{pr}}

\def\Frac{\operatorname{Frac}}
\def\ecdf{\mathcal D\operatorname{-CF}_0}
\def\ddcf{\mathcal D\operatorname{-DCF}_{0,m}}
\def\ddcfr{\mathcal D\operatorname{-DCF}_{0,r}}
\def\dcfa{\operatorname{DCF}_{0,m}\operatorname A}

\def\x{\overline x}
\def\a{\overline a}
\def\b{\overline b}

\def \mm {\mathfrak m}

\def\Ind#1#2{#1\setbox0=\hbox{$#1x$}\kern\wd0\hbox to 0pt{\hss$#1\mid$\hss}
\lower.9\ht0\hbox to 0pt{\hss$#1\smile$\hss}\kern\wd0}
\def\ind{\mathop{\mathpalette\Ind{}}}
\def\Notind#1#2{#1\setbox0=\hbox{$#1x$}\kern\wd0\hbox to 0pt{\mathchardef
\nn=12854\hss$#1\nn$\kern1.4\wd0\hss}\hbox to
0pt{\hss$#1\mid$\hss}\lower.9\ht0 \hbox to
0pt{\hss$#1\smile$\hss}\kern\wd0}

\begin{document}

\title[Differential fields with free operators]{The model companion of differential fields\\ with free operators}

\author{Omar Le\'on S\'anchez}
\address{Omar Le\'on S\'anchez\\
McMaster University\\
Department of Mathematics and Statistics\\
1280 Main Street West\\
Hamilton, Ontario \  L8S 4L8\\
Canada}
\email{oleonsan@math.mcmaster.ca}

\author{Rahim Moosa}
\address{Rahim Moosa\\
University of Waterloo\\
Department of Pure Mathematics\\
200 University Avenue West\\
Waterloo, Ontario \  N2L 3G1\\
Canada}
\email{rmoosa@uwaterloo.ca}

\date{\today}

\let\thefootnote\relax

\begin{abstract}
A model companion is shown to exist for the theory of partial differential fields of characteristic zero equipped with free operators that commute with the derivations.
The free operators here are those introduced in [R. Moosa and T. Scanlon, {\em Model theory of fields with free operators in characteristic zero}, Preprint 2013].
The proof relies on a new lifting lemma in differential algebra: a differential version of Hensel's Lemma for local finite algebras over differentially closed fields.
\end{abstract}

\maketitle

\tableofcontents

\section{Introduction}

\noindent
A new rather general formalism for the model theory of fields with operators was introduced and developed by the second author and Thomas Scanlon in~\cite{paperC}.
That theory treats differential fields and difference fields uniformly, and also allows for a wide variety of other additive operators, or systems of operators, whose multiplicative rules are induced  by a ring homomorphism from the field to a finite algebra over the field.
These were called {\em free} operators in~\cite{paperC}, and we delay recalling the precise definition until $\S$\ref{ddfields} below.
A model companion was shown to exist in characteristic zero, and both the basic model theory as well as versions of the Canonical Base Property and the Zilber Dichotomy for the model companion were established.
An important shortcoming of~\cite{paperC} is that the theories dealt with in that paper cannot impose any commutativity on the operators (this is partly why they are called ``free").
Indeed, at the level of generality of~\cite{paperC} it is known that commutativity cannot always be imposed: the theory of a pair of commuting automorphisms has no model companion.
Nevertheless, there are several natural examples of theories of commuting operators that have very well behaved and understood model companions.
In this paper we extend~\cite{paperC} to deal with some of these commuting situations; namely the context where the fields come already equipped with several commuting derivations and the new operators are assumed to commute with the derivations.
That is, we prove that {\em the theory of partial differential fields of characteristic zero equipped with free operators that commute with the derivations has a model companion}.
This is done in~$\S$\ref{ddclosedfields}.

A particular case is the theory of partial differential fields equipped with an automorphism.
The existence of the model companion for this theory was established by the first author in~\cite{leonsanchez13}, and served as a model for us.

There are two main obstacles in passing from fields to differential fields.
The first is that in order to to extend the geometric characterisation of the existentially closed models that appears in~\cite{paperC} one requires lifting nonsingular solutions of systems of differential polynomial equations from a differentially closed field to a local finite algebra over that field.
Such a differential analogue of Hensel's Lemma appears new and should be of independent interest.
It is established here as Theorem~\ref{diffhen}.

The second obstacle is that the geometric characterisation one obtains quantifies over irreducible differential varieties.
Because irreducibility of Kolchin closed sets is not known to be a definable property of the parameters, one has to modify the characterisation into a provably first-order axiomatisation of the existentially closed models.
Here we follow the method established by the first author in~\cite{leonsanchez12} and~\cite{leonsanchez13}, which uses characteristic sets of prime differential ideals.

A word about the proofs: we have tried to focus on those aspects of the arguments where the differential setting poses a significant challenge.
Whenever the results of~\cite{paperC} extend automatically, or by using well known standard methods, we either omit proofs or content ourselves with brief sketches.

 \bigskip
\section{Some Preliminaries on Differential Algebra}
\label{prelims}

\noindent
While we will be assuming familiarity with the model theory of differential fields and the differential algebra that it entails, we take this opportunity to fix some notation and also to recall some of the more technical aspects of differential algebra that may not be top of mind among model theorists.
For further details we suggest Chapters~I and IV of~\cite{kolchin73}.

Suppose $R$ is a unital and commutative ring (all of our rings will be so) equipped with $m$ commuting derivations $\Delta=\{\delta_1,\dots,\delta_m\}$.
For $x=(x_1,\dots,x_n)$ a tuple of indeterminates, we denote by $R\{x\}$ the $\Delta$-ring of $\Delta$-polynomials over $R$, and by $R\langle x\rangle$ the $\Delta$-field of $\Delta$-rational functions.
We denote by $\Theta$ the free commutative monoid  generated by $\Delta$, and consider the set of {\em algebraic indeterminates}
$$\T x=\{\t x_i:\, 1\leq i\leq n, \, \t\in \T\}$$
So the underlying $R$-algebra structure on $R\{x\}$ is that of the polynomial ring $R[\T x]$.
For $\Lambda\subseteq R\{x\}$ we will use $[\Lambda]$ to denote the $\Delta$-ideal and $(\Lambda)$ to denote the (usual algebraic) ideal, generated by $\Lambda$.

There is a canonical ranking on $\T x$ given by
\begin{displaymath}
\d_m^{e_m}\cdots \d_1^{e_1}x_i< \d_m^{r_m}\cdots \d_1^{r_1}x_j \iff \left(\sum e_k,i,e_m,\dots,e_1\right)<\left(\sum r_k,j,r_m,\dots,r_1\right)
\end{displaymath}
where the ordering on the right-hand-side is the lexicographic one.
This allows us to view $\T x$ as a sequence enumerated with respect to the above linear ordering.
Let $f\in R\{x\}\setminus R$. The \emph{leader} of $f$, $v_f$, is the highest ranking algebraic indeterminate that appears in $f$. The \emph{degree} of $f$, $d_f$, is the degree of $v_f$ in $f$. The \emph{rank} of $f$ is the pair $(v_f,d_f)$ under the lexicographic ordering.
By convention, an element of $R$ has lower rank than all the elements of $R\{x\}\setminus R$.
The \emph{separant} of $f$, $S_f$, is the formal partial derivative of $f$ with respect to $v_f$.
The \emph{initial} of $f$, $I_f$, is the leading coefficient of $f$ when viewed as a polynomial in $v_f$.
Note that both $S_f$ and $I_f$ have lower rank than $f$.
Given a finite subset $\Lambda\subset R\{x\}$, we set $H_\Lambda:=\prod_{f\in \Lambda}I_fS_f$.

\begin{definition}
\label{coherent}
A subset $\L\subset R\{x\}\setminus R$ is said to be {\em coherent} if the following properties are satisfied:
\begin{itemize}
\item[(1)] Autoreducedness.
For each $f\neq g$ in $\L$, no proper derivative of $v_f$ appears in $g$, and if $v_f$ appears at all in $g$ then it does so with strictly smaller degree.
\item[(2)] Coherency Condition.
Suppose that for some $f\neq g$ in $\L$, there are derivatives $\theta_f$ and $\theta_g$ such that $\theta_fv_f=\theta_gv_g=:v$.
Suppose moreover that $v$ is the least algebraic indetermiate for which this happens.
Then for some $N$, 
$$H_\L^N(S_g\theta_ff-S_f\theta_gg)\in(\L)_v$$
where $(\L)_v$ is the ideal generated by $\{\theta f: f\in\L, v_{\theta f}<v\}$.
\end{itemize}
\end{definition}

\begin{remark}
Autoreducedness implies that $\L$ is finite and that distinct elements of $\L$ have distinct leaders.
\end{remark}

\begin{lemma}\label{uselem}
 Let $\phi:R\to S$ be a homomorphism of differential rings.
 Suppose $\L$ is a coherent subset of $R\{x\}$, and let $\L^\phi:=\{f^\phi:f\in\L\}$, where $f^\phi\in S\{x\}$ is obtained by applying $\phi$ to the coefficients of $f$.
 If $H_\L^\phi\neq 0$ then $H_{\L^\phi}=H_\L^\phi$ and $\L^\phi$ is coherent.
\end{lemma}

\begin{proof}
In general, every algebraic indeterminate that appears in $f^\phi$ appears in $f$.
Moreover,
$$\frac{\partial f^\phi}{\partial v_f}=\left(\frac{\partial f}{\partial v_f}\right)^\phi=S_f^\phi$$
It follows that if $S_f^\phi\neq 0$, then $f$ and $f^\phi$ have the same leader and $S_{f^\phi}=S_f^\phi$.
In this case, if $I_f^\phi\neq 0$ we get also that $I_{f^\phi}=I_f^\phi$.

Since $H_\L^\phi\neq 0$, it follows from the above that $\L$ and $\L^\phi$ have the same leaders and that $H_{\L^\phi}=H_\L^\phi$.
Since the algebraic indeterminates of $f^\phi$ appear in $f$ with the same degree, we also get that $\L^\phi$ is autoreduced.
To see that the coherency condition is satisfied we need only observe that if $g\in (\L)_v$ then $g^\phi\in (\L^\phi)_v$.
\end{proof}

Suppose $K$ is a $\Delta$-field of characteristc zero.
While it is not the case that all prime $\Delta$-ideals of $K\{x\}$ are finitely generated as $\Delta$-ideals (though they are finitely generated as radical $\Delta$-ideals), something close is true; they are determined by certain canonical finite subsets. There is a natural ranking of the autoreduced subsets of $K\{x\}$ that is roughly lexicographical, see \cite[$\S$I.10]{kolchin73} for the definition. If $I\subset K\{x\}$ is a prime $\Delta$-ideal then a {\em characteristic set} $\L$ for $I$ is a minimal autoreduced subset of $I$ with respect to this ranking.
These always exist, are in addition coherent, and determine the ideal $I$ in the sense that
\begin{equation}\label{equ}
I=\{f\in K\{x\}:\text{ for some $\ell\geq 0$, } H_\Lambda^\ell f\in[\Lambda]\}
\end{equation}
Taking zero sets in some differentially closed extension of $K$  we have that
$$V(I)\setminus V(H_\Lambda)=V(\Lambda)\setminus V(H_\Lambda)$$
A consequence of the minimality of $\Lambda$ is that $H_\Lambda\notin I$, and hence the above equality says that $V(I)$ and $V(\L)$ agree off a proper Kolchin closed subset.

\bigskip
\section{A Differential Lifting Lemma}
\label{lift}

\noindent
The following theorem, which we will use later but which may also be of independent interest, is a differential analogue of the artinian case of Hensel's Lemma.

\begin{theorem}\label{diffhen}
Suppose $L$ is a differentially closed field and $R$ is a local finite $L$-algebra, equipped with a differential structure extending that of $L$, and such that the maximal ideal of $R$ is a $\Delta$-ideal.
Suppose $\L=\{f_1,\dots,f_s\}\subset R\{x\}$ is a coherent set of differential polynomials in one variable over $R$.
Suppose, finally, that $a\in L$ is such that $\res(f)(a)=0$ for all $f\in \L$ but $\res(H_\L)(a)\neq 0$. 
Then there exists $b\in R$ such that $\res(b)=a$ and $f(b)=0$ for all $f\in \L$.
\end{theorem}

\begin{remark}
A proof of this theorem in the special case  when $R=L[\epsilon]/(\epsilon^2)$ and $\epsilon$ is a $\Delta$-constant can be extracted from Kolchin~\cite[Chapter~0, $\S$4]{kolchin85}, and served as a model for us.
In the ordinary case, when $\Delta=\{\delta\}$, other special cases are also known.
For example the case of $R=L[\epsilon]/(\epsilon^{n+1})$ with $\epsilon$ a $\delta$-constant follows from Scanlon's $D$-Hensel's Lemma~\cite[Proposition~6.1]{scanlon2000}.
The main difficulty that our proof below must contend with is the presence of several commuting derivations.
\end{remark}

 \begin{proof}[Proof of~Theorem~\ref{diffhen}]
 Let $\mm$ be the maximal ideal of $R$.
 Note that as $L$ is algebraically closed, the residue field of $R$ is $L$.
 
 We will recursively define a sequence $b_1,b_2,\dots$ in $R$ such that
 \begin{itemize}
 \item
$\res(b_1)=a$
\item
$b_{i+1}\equiv b_i\mod\mm^i$, and 
\item
$f(b_i)\equiv 0\mod\mm^i$ for each $f\in\L$
\end{itemize}
As $\mm$ is nilpotent, this will produce, in finitely many steps, the desired root of $\L$ lifting $a$.
For the base case, we may take any $b_1\in R$ with $\res(b_1)=a$.
Suppose $i\geq 1$ and we have found $b_i$ with the desired properties.

As $\mm$ is a $\Delta$-ideal, and $R$ is a finite dimensional vector space over $L$, we have that $\mm^i/\mm^{i+1}$ is a finite dimensional {\em $\Delta$-module over $L$}.
That is, each $\delta_j$ induces an additive map on $\mm^i/\mm^{i+1}$, which we also denote by $\delta_j$, and which satisfies
$\delta_j(rv)=\delta_j(r)v+r\delta_j(v)$ for all $r\in L$ and $v\in\mm^i/\mm^{i+1}$.
The existence of fundamental systems of solutions to integrable linear differential equations in differentially closed fields \cite[Appendix D]{vdPSinger} implies that $\mm^i/\mm^{i+1}$ has an $L$-basis made up of $\Delta$-constants.
 We can therefore find $\epsilon_1,\dots,\epsilon_n\in \mm^i$ whose image is an $L$-basis for $\mm^i/\mm^{i+1}$, and such that  $\delta(\epsilon_k)\equiv 0\mod\mm^{i+1}$ for all $\delta\in\D$ and all $k=1,\dots,n$.
 
 Let $\lambda_1,\dots,\lambda_k:\mm^i\to L$ be the $L$-linear maps that satisfy
 $$v\equiv\sum_{k=1}^n\lambda_k(v)\epsilon_k\mod\mm^{i+1}$$
 for all $v\in\mm^i$.
 The fact that we have chosen the basis to be made up of $\Delta$-constants modulo $\mm^{i+1}$, implies that the $\lambda_k$'s will commute with $\Delta$.
 
 We are looking for $y_1,\dots,y_n\in L$ such that
 \begin{equation}
 \label{need}
 f(b_i+\sum_{k=1}^ny_k\epsilon_k)  \equiv  0\mod \mm^{i+1} \ \ \ \ \ \ \ \ \ \ \text{for all $f\in \L$}
\end{equation} 
Indeed, we could then set $b_{i+1}:=b_i+\sum y_k\epsilon_k$, and the construction would be complete.

To compute the left-hand-side of~(\ref{need}), we write $f(x)=F(\Theta x)$, where $F$ is an algebraic polynomial in the variables $\Theta x=(\theta x:\theta \in \Theta)$, and then apply a Taylor expansion to $F\big(\Theta b_i+\Theta (\sum_{k=1}^ny_k\epsilon_k)\big)$.
Note that as the $\epsilon_k$'s are $\Delta$-constants modulo $\mm^{i+1}$, $\Theta(\sum_{k=1}^ny_k\epsilon_k)\equiv\sum_{k=1}^n\Theta(y_k)\epsilon_k\mod\mm^{i+1}$.
Hence, the higher order terms in the Taylor expansion will vanish.
The upshot is that
\begin{eqnarray*}
f(b_i+\sum_{k=1}^ny_k\epsilon_k)
&\equiv &
f(b_i)+\sum_{\theta\in\Theta}\frac{\partial F}{\partial(\theta x)}(\Theta b_i)\cdot \big(\sum_{k=1}^n\theta(y_k)\epsilon_k\big)\mod\mm^{i+1}\\
&\equiv&
\sum_{k=1}^n\lambda_k(f(b_i))\epsilon_k+\sum_{k=1}^n\left(\sum_{\theta\in\Theta}\frac{\partial F}{\partial(\theta x)}(\Theta b_i)\cdot\theta(y_k)\right)\epsilon_k\mod\mm^{i+1}\\
&\equiv&
\sum_{k=1}^n\left[\lambda_k(f(b_i))+\sum_{\theta\in\Theta}\res\left(\frac{\partial F}{\partial(\theta x)}(\Theta b_i)\right)\theta(y_k)\right]\epsilon_k\mod\mm^{i+1}
\end{eqnarray*}
where in the last step we use the fact that for all $r\in R$, $r\epsilon_k\equiv \res(r)\epsilon_k\mod\mm^{i+1}$.
Hence, solving the system of congruences~(\ref{need}) is thereby reduced to solving, for each fixed $k=1,\dots,n$, the following inhomogeneous system of linear differential equations in $y_k$ over $L$:
\begin{equation}
\label{needeqn}
\lambda_k(f(b_i))+\sum_{\theta\in\Theta}\res\left(\frac{\partial F}{\partial(\theta x)}(\Theta b_i)\right)\theta(y_k) =0\ \ \ \ \ \ \text{for all $f\in \L$}
\end{equation}
Note that $i,b_i$, and $k$ are all fixed.

Let us denote the linear differential polynomial over $L$ that is the left-hand-side of~(\ref{needeqn}) by $\ell_f$.
So
$$\ell_f(y_k):=\lambda_k(f(b_i))+\sum_{\theta\in\Theta}\res\left(\frac{\partial F}{\partial(\theta x)}(\Theta b_i)\right)\theta(y_k)$$
To complete the proof of the theorem, therefore, we need to show that the linear differential system $\{\ell_f(y_k)=0:f\in \Lambda\}$ has a solution in $L$.
This would follow if we knew that $\ell\Lambda:=\{\ell_f:f\in\L\}$ was a coherent set in $L\{y_k\}$.
Indeed, any coherent set of linear differential polynomials generates a prime differential ideal (see for example statement~(2) on page~399 of~\cite{rosenfeld}), and hence the differentially closed $L$ would contain a root of $\ell\L$.
Actually the situation is a little more complicated as $\ell\L$ need not be coherent.
But we will find a coherent set of linear differential polynomials that generate the same differential ideal as $\ell\L$, and that will of course suffice for the above argument to go through.

Note that $\ell_f$ is defined for any $f\in R\{x\}$ with the property that $f(b_i)\in\mm^i$.
Consider the $\Delta$-ideal $V:=\{f\in R\{x\}:f(b_i)\in\mm^i\}$ of $R\{x\}$, and the map $\ell:V\to L\{y_k\}$ given by $f\mapsto\ell_f$.
By the induction hypothesis, $\Lambda\subset V$.

\begin{claim}
\label{ell}
The map $\ell:V\to L\{y_k\}$ satisfies the following properties:
\begin{itemize}
\item[(a)]
it is a homomorphism of $\Delta$-modules over $L$,
\item[(b)]
if $f\in V$ then, for any $g\in R\{x\}$, $\ell_{gf}=\res(g)(a)\ell_f$,
\item[(c)]
for $f\in V$, $f$ and $\ell_f$ have the same leaders (though in $x$ and $y_k$, respectively), and moreover, the separant of $\ell_f$ is $\res(S_f)(a)$.
\end{itemize}
\end{claim}

\begin{proof}[Proof of Claim~\ref{ell}]
That $\ell$ is $L$-linear is clear, so for part~(a) we need to show that for all $\delta\in\Delta$, $\delta(\ell_f)=\ell_{\delta f}$.
This can be directly verified from the definition of $\ell$, using the fact that $\lambda_k$ and $\res$ commute with $\delta$.
The computation is tedious but routine, and we leave it to the reader.

For part~(b), suppose $f\in V$.
By construction $\ell_{gf}(y_k)$ is the coefficient of $\epsilon_k$ in the expansion of
$$gf\left(b_i+\sum_{j=1}^ny_j\epsilon_j\right)=g\left(b_i+\sum_{j=1}^ny_j\epsilon_j\right)f\left(b_i+\sum_{j=1}^ny_j\epsilon_j\right)$$
modulo $\mm^{i+1}$.
We have already seen that for any $y_1,\dots,y_n\in L$, $f\left(b_i+\sum_{j=1}^ny_j\epsilon_j\right)$ is an $L$-linear combination of $\{\epsilon_1,\dots,\epsilon_n\}$ and hence lands in $\mm^i$.
It follows that, working modulo $\mm^{i+1}$, the only contribution of $g\left(b_i+\sum_{j=1}^ny_j\epsilon_j\right)$ to the above product will be $\res\left(g\left(b_i+\sum_{j=1}^ny_j\epsilon_j\right)\right)=\res(g)(a)$.
This proves part~(b).

Finally, for part~(c), note that if an algebraic indeterminate $\theta(y_k)$ appears in $\ell_f\in V$, then $\res\left(\frac{\partial F}{\partial(\theta x)}(\Theta b_i)\right)\neq 0$, and so $\theta x$ must appear in $f$.
On the other hand, if $v_f=\theta x$ is the leader of $f$, then the coefficient of $\theta(y_k)$ in $\ell_f$ is  $\res\big(\frac{\partial F}{\partial v_f}(\Theta b_i)\big)=\res\big(S_f(b_i) \big)=\res(S_f)(a)$, which is nonzero because by assumption $\res(H_\L)(a)\neq 0$.
It follows that $f$ and $\ell_f$ have the same leaders (though in $x$ and $y_k$, respectively).
\end{proof}

We use the claim to prove that $\ell\L:=\{\ell_f:f\in\Lambda\}\subset L\{y_k\}$ satisfies the coherency condition, that is, part~(2) of Definition~\ref{coherent}.
Suppose $f\neq g$ in $\L$ and there are derivatives $\theta_f$ and $\theta_g$ such that $\theta_fv_{\ell_f}=\theta_gv_{\ell_g}=:v$, and $v$ is the least such occurrence.
We show that $S_{\ell_g}\theta_f\ell_f-S_{\ell_f}\theta_g\ell_g\in(\ell\L)_v$.
Indeed, by~\ref{ell}(c), we have that $v=\theta_fv_f=\theta_gv_g$, and this is the least such occurrence for $\L$.
Since $\L$ is coherent, there is $N$ such that $h:=H_\L^N(S_g\theta_ff-S_f\theta_gg)\in(\L)_v$.
It follows by Claim~\ref{ell} parts~(a) and~(b) that $\ell_h\in(\ell\L)_v$.
But, also by~\ref{ell},
\begin{eqnarray*}
\ell_h
&=&
\res(H_\L)(a)\big(\res(S_g)(a)\theta_f\ell_f-\res(S_f)(a)\theta_g\ell_g\big)\\
&=&
\res(H_\L)(a)\big(S_{\ell_g}\theta_f\ell_f-S_{\ell_f}\theta_g\ell_g\big)
\end{eqnarray*}
Since $\res(H_\L)(a)\neq 0$, we get that $S_{\ell_g}\theta_f\ell_f-S_{\ell_f}\theta_g\ell_g\in(\ell\L)_v$, as desired.

It may not be the case, however, that $\ell\L$ is autoreduced.
The leaders of distinct members are distinct, and it is {\em partially} autoreduced in the sense that no proper derivative of the leader of $\ell_f$ appears in $\ell_g$ for $f\neq g$ in~$\L$.
Indeed, this follows from the fact that $\L$ is autoreduced, that $\ell_f$ and $f$ have the same leader, and that every algebraic indeterminate appearing in $\ell_g$ also appears in $g$.
The obstacle to autoreducedness is that the leader of $\ell_f$ may still appear in $\ell_g$, just not as its leader.
But now a standard reduction process (see the second paragraph of the proof of Proposition~5 of~\cite[Chapter~0, $\S$4]{kolchin85}, for example) can be used to transform $\ell\L$ into an autoreduced set.
The kinds of transformations we have in mind are as follows: if $v_{\ell_f}$ appears with coefficient $c$ in $\ell_f$ and coefficient $d$ in $\ell_g$, then we can replace $\ell_g$ by $\ell_g-\frac{d}{c}\ell_f$.
This process will produce a coherent set of linear differential polynomials $\Gamma\subset L\{y_k\}$ such that $[\Gamma]=[\ell\L]$.
It follows, as explained earlier, that $[\ell\L]$ is a prime differential ideal and so the system~(\ref{needeqn}) has a solution in $L$.
This completes the proof of Theorem~\ref{diffhen}.
\end{proof}

\begin{remark}
The above theorem remains true even if we remove the assumption that $R$ is local. In this case one replaces the maximal ideal by the Jacobson radical $\mathfrak J$ of $R$,  the residue map by the quotient map $\pr:R\to R/{\mathfrak J}$, and the condition $\res(H_\L)(a)\neq 0$ by $\pr(H_\L)(a)$ being a unit.
\end{remark}

\begin{proof}
The general case reduces to the local case of Theorem~\ref{diffhen} by the following claim:
{\em Suppose $(K,\Delta)\subseteq (R,\Delta)$ is an extension of differential rings where $K$ is a field of characteristic zero and $R$ is a finite $K$-algebra.
If $R=\prod_{i=0}^tR_i$ is the decomposition of $R$ into local $K$-algebras, then $\Delta$ induces a differential structure on each $R_i$.
Moreover, if the Jacobson radical of $R$ is a $\Delta$-ideal, then so is the maximal ideal of each $R_i$.}
To prove the above claim, for each $i=0,\dots, t$, and each $r\in R_i$, let $\widehat r$ denote the element of $R$ whose $R_i$-th co-ordinate is $r$ and whose $R_j$-th co-ordinate is $0_{R_j}$ for all $j\neq i$.
Let $\widehat R_i=\{\widehat r:r\in R_i\}$.
Then for all $\delta\in\Delta$ and $r\in R_i$ we have
$$\delta(\widehat r)=\delta(\widehat r\widehat 1_{R_i})=\widehat r\delta(\widehat 1_{R_i})+\delta(\widehat r)\widehat 1_{R_i}\in \widehat R_i$$
which shows that $\delta(\widehat R_i)\subseteq \widehat R_i$.
So we get an induced $\Delta$-structure on each $R_i$.
For the ``moreover" clause, note that if $\mathfrak m_i$ is the maximal ideal of $R_i$, then $\prod_{i=0}^t\mathfrak m_i$ is the Jacobson radical of $R$.

Note that the reduction to the local case also uses the fact that the co-ordinate projection $R\to R_i$ of a coherent set is again coherent.
This follows from Lemma~\ref{uselem}; one uses the assumption that $\pr(H_\L)(a)$ is a unit to see that the co-ordinate projection of $H_\Lambda$ is nonzero and so Lemma~\ref{uselem} applies.
\end{proof}

\bigskip
\section{$\cD$-Differential fields}
\label{ddfields}

\noindent
We wish to study differential fields $(K,\Delta)$ equipped with additional linear operators, or stacks of operators, that commute with the derivations in $\Delta$.
The behaviour of these operators with respect to multiplication in the field should be governed by a fixed finite $K$-algebra, in a way that we now make precise.

Fix throughout the rest of the paper the following data:
\begin{itemize}
\item
a base field $k$ of characteristic zero,
\item
a finite $k$-algebra $B$ with the property that for each of the (finitely many) maximal ideals $\mm$ of $B$, $B/\mm=k$,
\item
a $k$-algebra homomorphism $\pi:B\to k$, and
\item
a $k$-basis $(\epsilon_0,\dots,\epsilon_\ell)$ of $B$ such that $\pi(\epsilon_0)=1$ and $\pi(\epsilon_i)=0$ if $i\neq 0$.
\end{itemize}
Note that the condition on the residue fields above is automatically satisfied if $k$ is algebraically closed.

Given any $k$-algebra $R$ we let
$$\cD(R):=B\otimes_kR$$
denote the base extension, which will be a finite and free $R$-algebra whose corresponding basis we will also denote by $(\epsilon_0,\dots,\epsilon_\ell)$.
We can also lift $\pi:B\to k$ to a surjective $R$-algebra homomorphism $\pi^R:\cD(R)\to R$, by base extension.

Now, if $R$ is equipped with $k$-linear derivations $\Delta=\{\delta_1,\dots,\delta_m\}$, then we can canonically make $\cD(R)$ a $\Delta$-ring by defining $\delta_j$ on $\cD(R)$ to be the unique $B$-linear derivation that extends $\delta_j$ on $R$; that is, by defining $\delta_j(b\otimes r):=b\otimes\delta_j(r)$.

\begin{definition}
A {\em $\cD$-differential ring} is a differential $k$-algebra $(R,\Delta)$ equipped with a sequence of operators $\partial=(\partial_1,\dots,\partial_{\ell})$ such that $e:R\to \cD(R)$ given by
$$r\mapsto r\epsilon_0+\partial_1(r)\epsilon_1+\cdots+\partial_\ell(r)\epsilon_\ell$$
is a homomorphism of differential $k$-algebras.

An equivalent formulation of $(R,\Delta,\partial)$ being a $\cD$-differential ring is that:
\begin{itemize}
\item[(a)]
$R$ is a $k$-algebra
\item[(b)]
$\Delta=(\delta_1,\dots,\delta_m)$ are commuting $k$-linear derivations on $R$
\item[(c)]
$(R,\partial)$ is a $\cD$-ring in the sense of~\cite{paperC}; that is, the map $e:R\to\cD(R)$ defined above is a homomorphism of $k$-algebras, and
\item[(d)]
the $\partial$ commute with the $\Delta$.
\end{itemize}
Indeed, it is an easy computation to see that $e$ will preserve the differential structure if and only if each operator in $\partial$ commutes with each operator in $\Delta$.
\end{definition}

Thus, all the examples of free operators described in~$\S3$ of~\cite{paperC} have $\cD$-differential analogues by imposing the further condition that they commute with $\D$.
In particular, difference-differential rings are $\cD$-differential where $\cD$ is given by setting $B=k\times k$ and  taking the standard basis.

\begin{remark}
As in~\cite{paperC}, we often identify a $\cD$-differential structure on $(R,\Delta)$ with a differential $k$-algebra homomorphism $e:R\to\cD(R)$ having the property that $\pi^R\circ e=\id$.
Indeed, given such an $e$, if we let $\partial_i$ be the operator that takes $r$ to the $\epsilon_i$-coefficient of $e(r)$ in $\cD(R)$, then $(R,\Delta,\partial)$ is a $\cD$-differential ring.
\end{remark}

Associated to a $\cD$-differential ring $(R,\Delta,\partial)$ we have canonical endomorphisms of $(R,\Delta)$.
Indeed, these are precisely the endomorphisms of $R$ associated to $(R,\partial)$ in~\cite[$\S4$]{paperC}.
To describe them, let $\displaystyle B=\prod_{i=0}^tB_i$ be the decomposition of $B$ into local $k$-algebras.
By assumption the residue field of each $B_i$ is $k$, and so we get surjective $k$-algebra homomorphisms $\pi_i:B\to k$ obtained by precomposing the residue map $B_i\to k$ with the co-ordinate projection $B\to B_i$.
Now, given a $k$-algebra $R$, we have the corresponding decomposition $\displaystyle \cD(R)=\prod_{i=0}^tB_i\otimes_kR$, and surjective $R$-algebra homomorphisms $\pi_i^R:\cD(R)\to R$.
Note that because the maximal ideal of $B_i$ is nilpotent, if $L$ is any  field extension of $k$ then $B_i\otimes_kL$ is again a local $L$-algebra with residue field $L$.
Hence in this case $\pi_i^L$ is the composition of the co-ordinate projection $\DD(L)\to B_i\otimes_k L$ with the residue map $B_i\otimes_kL\to L$.

If $(R,\partial)$ is a $\cD$-ring then the {\em associated endomorphisms} of $R$ introduced in~\cite{paperC} are the maps $\sigma_i:=\pi_i^R\circ e$.
If, moreover, $(R,\Delta,\partial)$ is a $\cD$-differential ring, then $\pi_i^R$ and $e:R\to \cD(R)$ both preserve the differential structure, and so the $\sigma_i$ are endomorphisms of $(R,\Delta)$.
We may, and do, assume that $\pi_0=\pi$, so that $\sigma_0=\id$.

Extending $\cD$-differential structure will entail extending the associated endomorphisms.
In order to pass from an integral domain to the fraction field it is therefore necessary that the associated endomorphisms be injective. In fact this is sufficient:

\begin{lemma}
\label{fraction}
Suppose $R$ is a $\Delta$-subring of a $\Delta$-field $(L,\Delta)$, extending $k$.
Suppose $e:R\to\cD(L)$ is a $k$-linear homomorphism of $\Delta$-rings, such that for all $i=0,1,\dots,t$, the maps $\pi_i^L\circ e:R\to L$ are injective.
Then there is a unique extension of $e$ to $\big(\Frac(R),\Delta\big)$.
\end{lemma}

\begin{proof}
Forgetting about $\Delta$ for the moment, by Lemma~4.9 of~\cite{paperC}, we can extend $e$ to a $k$-algebra homomorphism on $\Frac(R)$.
But such an extension is unique and automatically preserves the differential structure.
\end{proof}

The following proposition reduces the problem of extending $\cD$-differential structure to the much easier problem of extending difference--differential structure.
It is in here that our differential lifting lemma (Theorem~\ref{diffhen}) is used.

\begin{proposition}
\label{extension}
Suppose $(R,\Delta)$ is a $\Delta$-subring of a differentially closed field $(L,\Delta)$ extending $k$, and $e:R\to\cD(L)$ is a homomorphism of differential $k$-algebras such that $\pi^L\circ e:R\to L$ is the inclusion map.
Suppose moreover that we have extensions of $\pi_1^L\circ e,\dots,\pi_t^L\circ e$ to endomorphisms of $(L,\Delta)$, say $\sigma_1,\dots,\sigma_t$.
Then there exists an extension of $e$ to a $\cD$-differential field structure on $L$ whose associated endomorphisms are $(\sigma_1,\dots,\sigma_t)$.
\end{proposition}

\begin{proof}
The assumptions imply that the $\pi_i^L\circ e$ are all injective on $R$.
We can therefore extend $e$ canonically to $K:=\Frac(R)$ by Lemma~\ref{fraction}.

Suppose that for arbitrary $a\in L$ we can extend $e$ to a homomorphism on the $\Delta$-ring generated by $a$ over $K$, denoted by $K\{a\}$, in such a way that $\pi_i^L\circ e(a)=\sigma_i(a)$ for all $i=0,\dots,t$.
It would then follow that each $\pi_i^L\circ e$ agrees with $\sigma_i$ on all of $K\{a\}$, and hence is in particular injective.
By Lemma~\ref{fraction} again, $e$ would extend to the $\Delta$-field generated by $a$ over $K$, denoted $K\langle a\rangle$, and $\pi_i^L\circ e$ would agree with $\sigma_i$ on all of $K\langle a\rangle$.
By a Zorn's Lemma argument, using the arbitrariness of $a$, this would complete the proof of the Proposition.

It suffices, therefore, to prove that for arbitrary $a\in L$, $e$ extends to $K\{a\}$ in such a way that $\pi_i^L\circ e(a)=\sigma_i(a)$ for all $i=0,\dots,t$.
Suppose first that $a$ is $\Delta$-transcendental over $K$.
Let $b_i\in B_i\otimes _k L$ be any element with residue $\sigma_i(a)$.
Now, by $\Delta$-transcendentality, there is a unique extension of $e$ to a homomorphism on $(K\{a\},\Delta)$ which takes $a$ to $(b_0,b_1,\dots,b_t)$.
By construction, $\pi_i^L\circ e(a)=\sigma_i(a)$.

It remains to consider the case when $a$ is $\Delta$-algebraic over $K$.
In order to extend $e$ to a homomorphism on $(K\{a\},\Delta)$ we need to find $b\in\cD(L)$ such that $g^e(b)=0$ for all $g\in I_\D(a/K):=\{f\in K\{x\}:f(a)=0\}$.
(Recall that $g^e$ denotes the differential polynomial over $\cD(L)$ obtained from $g$ by applying $e$ to the coefficients.)
Moreover, we require that $\pi_i^L(b)=\sigma_i(a)$ for all $i=0,\dots,t$.
The desired extension would then send $a$ to $b$.
Using the decomposition $\displaystyle \cD(L)=\prod_{i=0}^t\big(B_i\otimes_kL)$, it suffices to show:
\begin{itemize}
\item[($*$)]
Fix an arbitrary $i\in\{0,\dots,t\}$, let $R_i:=B_i\otimes_kL$ equipped with the unique $B_i$-linear differential structure extending $\Delta$ on $L$, and let $e_i:K\to R_i$ be the composition of $e$ with the projection onto the $i$th factor.
Then there exists $b_i\in R_i$ such that $g^{e_i}(b_i)=0$ for all $g\in I_\D(a/K)$, and such that $\res_i^L(b_i)=\sigma_i(a)$.
\end{itemize}
Indeed, $b=(b_0,\dots,b_t)$ would then work.

We show~($*$) by first lifting $\sigma_i(a)$ to a root of $\L^{e_i}$ in $R_i$, where $\L$ is a characteristic set for $I_\D(a/K)$.
To do this we need to verify the hypotheses of Theorem~\ref{diffhen}.
Note that $L$ is differentially closed by assumption (and this is the only place we use that assumption), and $R_i$ is a local finite differential $L$-algebra whose maximal ideal is a $\Delta$-ideal as it is generated by the maximal ideal of $B_i$ which are all $\Delta$-constants. 
If $f\in\L$ then $f(a)=0$, so that applying $\sigma_i$ shows that $\res_i^L(f^{e_i})$ vanishes on $\sigma_i(a)$.
On the other hand, since $H_\L(a)\neq 0$, applying $\sigma_i$ shows that $\res_i^L(H^{e_i}_\L)$ does not vanish on $\sigma_i(a)$.
In particular, $H^{e_i}_\L$ is not zero in $R_i\{x\}$ and so, by Lemma~\ref{uselem}, $H_{\L^{e_i}}=H_\L^{e_i}$ and $\L^{e_i}$ is coherent.
Theorem~\ref{diffhen} therefore applies, and we have $b_i\in R_i$ with residue $\sigma_i(a)$ and such that $g^{e_i}(b_i)=0$ for all $g\in\L$.
It follows that $g^{e_i}(b_i)=0$ for all $g\in [\L]$.
If $g\in I_\D(a/K)$ then, by (\ref{equ}) of \S\ref{prelims}, we have that $H_\L^Ng\in[\L]$ for some $N$, and so $\big(H_{\L^{e_i}}(b_i)\big)^Ng^{e_i}(b_i)=0$.
But we know that $H_{\L^{e_i}}(b_i)\neq 0$, since its residue is $\sigma_i\big(H_\L(a)\big)\neq 0$.
It follows that $g^{e_i}(b_i)=0$ for all $g\in I_\D(a/K)$.
This completes the proof of~($*$), and hence of Proposition~\ref{extension}.
\end {proof}

\bigskip
\section{$\cD$-prolongations of $\Delta$-varieties}
\label{prolong}

\noindent
Given a $\cD$-differential field $(K,\Delta,\partial)$ and an affine $\Delta$-variety $V\subseteq\mathbb A^n$ over $K$ we wish to define the $\cD$-prolongation of $V$.
Taking our cue from the general theory of abstract prolongations of algebraic varieties developed in~\cite{paperA}, the construction should look something like this: base change $V$ to $\cD(K)$ via the homomorphism $e:K\to\cD(K)$ and then take the Weil restriction back down to $K$ via the standard $K$-algebra structure on $\cD(K)=K\otimes_kB$.
Since the theory of Weil restrictions is not to our knowledge developed in the differential context, we will instead give this construction explicitly in co-ordinates.

Let $x:=(x_1,\dots,x_n)$ be co-ordinates for $\mathbb A^n$, and
$\x:=\big(x=x^{(0)},x^{(1)},\dots,x^{(\ell)}\big)$ co-ordinates for $\mathbb A^{(\ell+1) n}$, where recall $\ell$ is such that that $\dim_kB=\ell+1$.
Given $f\in K\{x\}$, let $f^{(0)},\dots,f^{(\ell)}\in K\{\x\}$ be such that
\begin{equation}
\label{fi}
f^e\big(\sum_{j=0}^\ell x^{(j)}\epsilon_j\big)=\sum_{j=0}^{\ell} f^{(j)}(\x)\epsilon_j
\end{equation}
in $\DD(K)\{\x\}$.
Note that $\DD(K)\{\x\}=\cD(K\{\x\})$ is a free $K\{\x\}$-module with basis $\{\epsilon_0,\dots,\epsilon_\ell\}$, so that these $f^{(j)}$ are uniquely determined.

Fix a differentially closed field $(L,\Delta)$ extending $(K,\Delta)$.

\begin{definition}
Suppose $V\subseteq L^n$ is a $\Delta$-closed set defined over $K$.
The {\em $\cD$-prolongation} of~$V$, denoted by $\tau(V,\cD,e)$, or just $\tau V$ for short, is the $\Delta$-closed subset of $L^{(\ell+1)n}$ defined by the equations $f^{(j)}(\x)=0$ for all $f\in I_\Delta(V/K)$ and all $j=0,\dots,\ell$.
\end{definition}

\begin{remark}
\label{prorem} \
\begin{itemize}
\item[(a)]
A straightforward computation shows that if $I_\Delta(V/K)=[\Lambda]$ then in defining $\tau V$ it suffices to consider $f\in\Lambda$.
It follows that $\tau$ commutes with base change, and so preserves $\DD$-differential fields of definition.
\item[(b)]
The association $\a\mapsto \sum_{j=0}^\ell a^{(j)}\epsilon_j$ gives us a natural identification of $\tau V$ with the solutions in $\cD(L)^n$ of the set of equations $\{f^e(x)=0:f\in I_{\Delta}(V/K)\}$.
\item[(c)]
In the pure $\cD$-fields case when $\Delta=\emptyset$, the prolongations introduced here coincide with the $\cD$-prolongations introduced in~\cite{paperA} and used in~\cite{paperC}.
\item[(d)]
When $\cD(R)=R\times R$, the prolongation introduced here is the difference-differential prolongation $V\times V^{\sigma}$ that was used in \cite{leonsanchez13}.
\end{itemize}
\end{remark}

Recall that our basis elements were chosen so that the projection onto the $\epsilon_0$-co-ordinate is the ring homomorphism $\pi:B\to k$.
It follows that the coefficient of $\epsilon_0$ in $\epsilon_j\epsilon_k$ is $0$ unless $j=k=0$ in which case it is $1$.
This, together with the fact that the $\epsilon_j$ are  $\Delta$-constants, implies that under the identification of $x^{(0)}=x$, $f^{(0)}(\x)=f(x)$.
Hence, the $x^{(0)}$-coordinate projection maps $\tau V$ to $V$.
We denote this morphism of $\Delta$-closed sets by $\widehat\pi:\tau V\to V$.

On $K$-points, $\widehat\pi$ has a section denoted by $\nabla:V(K)\to \tau V(K)$, and given by $a\mapsto(a,\partial_1a,\dots,\partial_\ell a)$.
Indeed, if $a\in V(K)$ and $f\in I_\Delta(V/K)$, then computing in $\cD(K)$
$$0=e(f(a))=f^e(a+\partial_1a+\cdots+\partial_\ell a)=\sum_{j=0}^\ell f^{(j)}(a,\partial_1a,\dots,\partial_\ell a)\epsilon_j$$
so that each $f^{(j)}(\nabla a)=0$, as desired.

Finally, for each $i=0,\dots,t$, we also have natural morphisms $\widehat\pi_i:\tau V\to V^{\sigma_i}$, where $V^{\sigma_i}$ refers to the transform of $V$ obtained by applying the associated endomorphism $\sigma_i$ to the coefficients of the defining equations of $V$.
This map is the one given on $L$-points by $\displaystyle \overline a\mapsto \sum_{j=0}^\ell a^{(j)}\pi_i(\epsilon_j)$, where recall that $\pi_i$ is the composition of the $k$-algebra projection $B\to B_i$ in the decomposition $B=\prod_{i=0}^tB_i$, with the residue map $B_i\to k$.
That this works is a consequence of the following easy computation:

\begin{lemma}
\label{fsigma}
Suppose $f\in K\{x\}$ and $\a\in L^{(\ell+1)n}$ is such that $f^{(j)}(\a)=0$ for all $j=0,\dots,\ell$.
Then for each $i=0,\dots,t$, $\sum_{j=0}^\ell a^{(j)}\pi_i(\epsilon_j)$ is a root of $f^{\sigma_i}$.
\end{lemma}

\begin{proof}
We compute in $L$,
$$0=\pi_i^L\big(\sum_{j=0}^\ell f^{(j)}(\a)\epsilon_j\big)=\pi_i^L\left(f^e\big(\sum_{j=0}^\ell a^{(j)}\epsilon_j\big)\right)=f^{\sigma_i}\big(\sum_{j=0}^\ell a^{(j)}\pi_i(\epsilon_j)\big)$$
where the second equality is by~(\ref{fi}) and the third is because $\pi_i^K\circ e=\sigma_i$.
\end{proof}

One checks easily that $\widehat\pi_0=\widehat\pi$.
Note also that $\widehat\pi_i\circ \nabla(a)=\sigma_i(a)$, for all $i$.

\bigskip
\section{$\cD$-differentially closed fields}
\label{ddclosedfields}

\noindent
The language in which we study $\cD$-differential rings is that of $k$-algebras equipped with function symbols for the derivations in $\Delta$ and the operators in $\partial$.
Note that the associated endomorphisms are $0$-definable as they are in fact $k$-linear combinations of the operators in $\partial$, as was pointed out in~\cite[$\S$4.1]{paperC}.

Let $\cK$ denote the class of $\cD$-differential rings that are integral domains and for which the associated endomorphisms are injective.
By Lemma~\ref{fraction}, $\cK$ is precisely the set of models of the universal part of the theory of $\cD$-differential fields.
Here is a geometric characterisation of existentially closed $\cD$-differential fields that precisely extends Theorem~4.6 of~\cite{paperC} to the differential setting.

\begin{theorem}
\label{geoax}
A $\cD$-differential ring $(K,\Delta,\partial)$ is existentially closed in the class $\cK$ if and only if
\begin{itemize}
\item[I]
$(K,\Delta)$ is a differentially closed field,
\item[II]
the associated endomorphisms are automorphisms, and
\item[III]
if $V\subseteq K^n$ is an irreducible $\Delta$-closed set and $W\subseteq \tau V$ is an irreducible $\Delta$-closed subset of the $\cD$-prolongation such that $\widehat\pi_i(W)$ is $\Delta$-dense in $V^{\sigma_i}$ for all $i=0,\dots,t$, then there exists $a\in V$ with $\nabla(a)\in W$.
\end{itemize}
\end{theorem}

\begin{proof}
The only aspect of the proof of~\cite[Theorem~4.6]{paperC} that does not readily extend to the differential context is the reliance on the extension theorems which in turn rely on a Hensel's Lemma for local finite algebras over algebraically closed fields.
So replacing those arguments in~\cite{paperC} with our extension theorem (Proposition~\ref{extension} above), which followed from our differential Hensel's Lemma (Theorem~\ref{diffhen} above), we get a proof Theorem~\ref{geoax} quite easily.
Let us give a few details.

Suppose $(K,\Delta,\partial)$ is existentially closed in $\cK$.
Now, as the associated endomorphisms are isomorphisms between differential integral domains, and since $\operatorname{DCF}_{0,m}$ is the model completion of the theory of differential integral domains, the $\sigma_1,\dots,\sigma_t$ can be extended to automorphisms of some differentially closed field $(L,\Delta)$, which we also denote by $\sigma_1,\dots,\sigma_t$.
By Proposition~\ref{extension}, we can then extend the $\cD$-structure on $K$ to $L$ so that $(L,\Delta,\partial)$ is a $\cD$-differential field whose associated endomorphisms are the automorphisms $\sigma_1,\dots,\sigma_t$.
As $(K,\Delta,\partial)$ is existentially closed, it follows that $(K,\Delta)$ is already a differentially closed field and the associated endomorphisms are already surjective, thus establishing~I and~II.

For condition~III, let $(L,\Delta)$ be a differentially closed field extending $(K,\Delta)$ such that $W(L)$ has a point $\a$ that is Kolchin generic over $K$.
So $\a\in\tau V(L)$ and $a^{(0)}\in V(L)$.
Letting $a'=\sum_{j=0}^\ell a^{(j)}\epsilon_j$, we see that $a'$ is a root in $\cD(L)^n$ of $f^e(x)$ for all $f\in I_{\Delta}(V/K)$; see Remark~\ref{prorem}(b).
By the $\Delta$-density of $\widehat\pi(W)$ in $V$, we have that $I_\Delta(a^{(0)}/K)=I_\Delta(V/K)$, and so we can extend $e:K\to \cD(L)$ to $K\{a^{(0)}\}$ by sending $a^{(0)}\mapsto a'$.
Just as in~\cite[Theorem~4.6]{paperC}, the $\Delta$-density of $\widehat\pi_i(W)$ in $V^{\sigma_i}$, for $i=1,\dots,t$, will imply that the endomorphisms associated to $e:K\{a^{(0)}\}\to\cD(L)$ are injective.
Now, as in the previous paragraph, extending these endomorphisms to automorphisms of some $(L',\Delta)\supseteq (L,\Delta)$ and then using Proposition~\ref{extension}, we obtain a $\cD$-differential structure on $L'$ extending $e$ on $K\{a^{(0)}\}$.
So $(L',\Delta,\partial)\supseteq (K,\Delta,\partial)$ and $\nabla(a^{(0)})=\a$.
As $(K,\Delta,\partial)$ is existentially closed, it follows that there must exist $a\in V(K)$ with $\nabla(a)\in W(K)$, as desired.

As the readers can verify for themselves, the converse direction of Theorem~\ref{geoax} is proved exactly as the corresponding direction of~\cite[Theorem~4.6]{paperC}, with algebraically closed fields replaced by differentially closed fields and the Zariski topology replaced by the Kolchin topology.
\end{proof}

\begin{remark}
\label{better3}
Inspecting the proof of property~III for $(K,\Delta,\partial)$ existentially closed reveals a stronger conclusion: given any nonempty $\Delta$-open set $O\subseteq W$ the $a\in V$ could be found so that $\nabla(a)\in O$.
The point is that $\a\in W(L)$, being generic over~$K$, must land in $O(L)$.
\end{remark}

As in the case of existentially closed difference-differential fields considered in~\cite{leonsanchez13}, we do not know if the above characterisation amounts to a first-order axiomatisation because in differentially closed fields the property of being Kolchin irreducible is not known to be a definable property of the parameters of a $\Delta$-closed set.
Our strategy to avoid this difficulty is precisely the same as that of the first author in~\cite{leonsanchez13}.

Recall that for $\Lambda$ a finite subset of a $\Delta$-polynomial ring over a $\Delta$-field, $H_\Lambda$ is the product of the initials and separants of the members of $\Lambda$.
As was explained in~$\S$\ref{prelims}, when $\Lambda$ is a characteristic set for a prime $\Delta$-ideal $I$, the zero sets of $\Lambda$ and $I$ agree off the proper $\Delta$-closed hypersurface defined by $H_\Lambda$.
It will be useful to introduce notation for this $\Delta$-open set, we denote it by $V^*(\Lambda):=V(\Lambda)\setminus V(H_\Lambda)$.

\begin{proposition}
\label{3prime}
Suppose $(K,\Delta,\partial)$ is a $\cD$-differential field where $(K,\Delta)$ is differentially closed.
Then condition~III of Theorem~\ref{geoax} is equivalent to:
\begin{itemize}
\item[III$^\prime$]
Let $x=(x_1,\dots,x_n)$ be co-ordinates for $K^n$ and $\x=\big(x^{(0)},x^{(1)},\dots,x^{(\ell)}\big)$ co-ordinates for $\tau(K^n)=K^{(\ell+1)n}$.
Suppose $\Lambda\subset K\{x\}$ and $\Gamma\subset K\{\x\}$ are characteristic sets for prime ideals such that
\begin{itemize}
\item[(i)]
in $\tau(K^n)$, $V^*(\Gamma)\subseteq V\big(\{f^{(j)}:j=0,\dots,\ell, f\in \Lambda\}\big)$, and 
\item[(ii)]
in $K^n$, $\widehat\pi_i\big(V^*(\Gamma)\big)$ contains a nonempty $\Delta$-open subset of $V^*(\Lambda^{\sigma_i})$, for all $i=0,\dots, t$.
\end{itemize}
Then there exists $a\in V^*(\Lambda)$ with $\nabla(a)\in V^*(\Gamma)$.
\end{itemize}
\end{proposition}

\begin{proof}
Note that the $\widehat\pi_i$ referred to in~(ii) is $\widehat\pi_i:\tau(K^n)\to (K^n)^{\sigma_i}=K^n$.

Assume that III$^\prime$ holds, and suppose $V$ and $W$ are as in~III of Theorem~\ref{geoax}.
Let $\Lambda$ be a characteristic set for $I_{\Delta}(V/K)$ and $\Gamma$ a characteristic set for $I_\Delta(W/K)$.
The fact that $W\subseteq \tau(V)$ and $V^*(\Gamma)$ is contained in $W$ implies assumption~(i).
Since $\Lambda^{\sigma_i}$ is a characteristic set for $V^{\sigma_i}$, $V^*(\Lambda^{\sigma_i})$ is $\Delta$-open in $V^{\sigma_i}$.
Hence the fact that $\widehat\pi_i(W)$ is $\Delta$-dense in $V^{\sigma_i}$ implies, using quantifier elimination for differentially closed fields, the truth of~(ii).
Therefore there exists $a\in V^*(\Lambda)\subseteq V$ with $\nabla(a)\in V^*(\Gamma)\subseteq W$, as desired.

For the converse we assume~III and prove~III$^\prime$.
Suppose $\Lambda$ and $\Gamma$ are given.
Let $V\subset K^n$ denote the irreducible $\Delta$-closed set defined by the prime $\Delta$-ideal that has $\Lambda$ as a characteristic set, and let $W\subseteq\tau(K^n)$ be the irreducible $\Delta$-closed set defined by the prime $\Delta$-ideal that has $\Gamma$ as a characteristic set.
Since $W\setminus V(H_\Gamma)=V^*(\Gamma)$, assumption~(i) implies that
$W\subseteq V\big(\{f^{(j)}:j=0,\dots,\ell, f\in \Lambda\}\big)$.
But because $I_\Delta(V/K)$ need not be $[\Lambda]$, this does not imply that $W\subseteq\tau V$.
Let us assume for now, however, that $W\subseteq\tau V$.
Since $V^*(\Lambda^{\sigma_i})$ is $\Delta$-open in $V^{\sigma_i}$, assumption~(ii) does imply that $\widehat\pi_i(W)$ is $\Delta$-dense in $V^{\sigma_i}$.
We can apply~III of Theorem~\ref{geoax}, or rather the refinement mentioned in Remark~\ref{better3}, to get $a\in V^*(\Lambda)$ with $\nabla(a)\in V^*(\Gamma)$.

It therefore remains only to prove that $W\subseteq \tau V$.
Since we know that $W\subseteq V\big(\{f^{(j)}:j=0,\dots,\ell, f\in \Lambda\}\big)$, Lemma~\ref{fsigma} implies that $\widehat\pi_i(W)\subseteq V(\Lambda^{\sigma_i})$ for all $i=0,\dots, t$.
If we therefore let $O_i$ be the nonempty $\Delta$-open subset of $V(\Lambda^{\sigma_i})$ whose existence is asserted by~(ii), then $\widehat\pi_i^{-1}(O_i)\cap W$ is a nonempty $\Delta$-open subset of $W$.
Hence $\displaystyle O:=W\cap \bigcap_{i=0}^t\widehat\pi_i^{-1}(O_i)$ is a nonempty $\Delta$-open subset of $W$.
So to prove that $W\subseteq\tau V$ it suffices to prove that $O\subseteq\tau V$.
Let $\b\in O$.
To show that $\b\in\tau V$ we need to show that $f^{(j)}(\b)=0$ for all $f\in I_\Delta(V/K)$ and $j=0,\dots,\ell$.
We know this is true for all $f\in\Lambda$ since $\b\in W$, and it is easy to see that the same must hold for all $f\in[\Lambda]$.
But as $\Lambda$ is characteristic of $I_\Delta(V/K)$, if $f\in I_\Delta(V/K)$ then $H_\Lambda^Nf\in[\Lambda]$ for some~$N$, see equality (\ref{equ}) of \S\ref{prelims}.
Hence, using identity~(\ref{fi}) of~$\S$\ref{prolong}, we get that for all $f\in I_\Delta(V/K)$
\begin{eqnarray*}
0
&=&\sum_{j=0}^\ell (H_\Lambda^Nf)^{(j)}(\b)\epsilon_j\\
&=&(H_\Lambda^Nf)^e(\sum_{j=0}^\ell b^{(j)}\epsilon_j)\\
&=&(H_\Lambda^N)^e(\sum_{j=0}^\ell b^{(j)}\epsilon_j) \cdot f^e(\sum_{j=0}^\ell b^{(j)}\epsilon_j)\\
&=& (H_\Lambda^N)^e(\sum_{j=0}^\ell b^{(j)}\epsilon_j) \cdot (\sum_{j=0}^\ell f^{(j)}(\b)\epsilon_j)
\end{eqnarray*}
It therefore suffices to verify that $(H_\Lambda^N)^e(\sum_{j=0}^\ell b^{(j)}\epsilon_j)$ is a unit in $\cD(K)$, or equivalently that for each $i=0,\dots,t$, $\pi^K_i(H_\Lambda^N)^e(\sum_{j=0}^\ell b^{(j)}\epsilon_j)\neq 0$ in $K$.
But
$$\pi^K_i(H_\Lambda^N)^e(\sum_{j=0}^\ell b^{(j)}\epsilon_j)=(H_\Lambda^N)^{\sigma_i}(\sum_{j=0}^\ell b^{(j)}\pi_i\epsilon_j)=(H_{\Lambda^{\sigma_i}}^N)(\widehat \pi_i\b)$$
and the latter is not zero because $\widehat\pi_i(\b)\in O_i$ and $O_i$ is disjoint of $V(H_{\Lambda^{\sigma_i}})$ by assumption.
This completes the proof that $O\subseteq\tau V$, and hence that $W\subseteq\tau V$.
\end{proof}

\begin{corollary}
The theory of $\cD$-differential fields has a model companion.
\end{corollary}

\begin{proof}
This is because conditions~I and ~II of Theorem~\ref{geoax}, and~III$^\prime$ of Proposition~\ref{3prime}, are first-order  axiomatisable. For the latter the main point is that, by a criterion of Rosenfeld's, being the characteristic set of a prime $\Delta$-ideal is definable in parameters.
See~\cite{leonsanchez13} for a more detailed discussion.
\end{proof}

We call this model companion the theory of {\em $\cD$-differentially closed fields} and denote it by $\ddcf$.
When $m=0$ (so $\Delta=\emptyset$) we recover the theory of existentially closed $\cD$-fields, $\ecdf$, introduced in~\cite{paperC}.
On the other hand, when $\cD(R)=R\times R$ we recover the theory of existentially closed difference-differential fields $\dcfa$ whose existence was established in~\cite{leonsanchez13}.

\bigskip
\section{Properties of $\ddcf$}
\label{prope}

\noindent
Given that the model companion exists it is by now routine to show that this model companion is tame in the various ways one expects.
We list below five basic model-theoretic properties of $\ddcf$ that are readily verified using the results and methods of~\cite[$\S$5]{paperC}, which were themselves based on the work of Chatzidakis and Hrushovski on difference fields~\cite{acfa1}.
As no serious complications arise in the differential setting, we omit proofs entirely.
(See also~\cite{leonsanchez13} for the difference-differential case when $\cD(R)=R\times R$.)

We will need a bit more notation. A $\cD$-differential field is said to be {\em inversive} if the associated endomorphisms are surjective.
Given a subset $B$ of a model of $\ddcf$ we denote by $\langle B\rangle$ the smallest inversive $\cD$-differential subfield that contains $B$.

\begin{itemize}
\item[1.]
\underline{Completions}.
The completions of $\ddcf$ are determined by the action of the associated endomorphisms on $k^{\alg}$.
In fact, if $(K,\Delta,\partial)$ and $(L,\Delta',\partial')$ are models of $\ddcf$ with a common algebraically closed  inversive $\cD$-differential subfield $F$, then $(K,\Delta,\partial)\equiv_F(L,\Delta',\partial')$.
\item[2.]
\underline{$\acl$}.
If $B$ is a subset of a model of $\ddcf$ then $\acl(B)=\langle B\rangle^{\alg}$.
In the case when $t=0$ (so $B$ is a local $k$-algebra), $\dcl(B)=\langle B\rangle$.
\item[3.]
\underline{Quantifier reduction}.
Suppose $(L,\Delta,\partial)\models\ddcf$, $K$ is an inversive $\cD$-differential subfield, and $a,b\in L^n$.
Then $\tp(a/K)=\tp(b/K)$ if and only if there is an isomorphism between $(K\langle a\rangle,\Delta,\partial)$ and $(K\langle b\rangle,\Delta,\partial)$ that fixes $K$, takes $a$ to $b$, and extends to an isomorphism from $(K\langle a\rangle^{\alg},\sigma)$ to $(K\langle b\rangle^{\alg},\sigma)$.
In particular, if $t=0$ then $\ddcf$ admits quantifier elimination.
\item[4]
\underline{Simplicity}.
Every completion of $\ddcf$ is simple with the independence theorem holding over (model-theoretically) algebraically closed sets.
Nonforking independence is given by $A\ind_CB$ if and only if $\langle A\rangle$ and $\langle B\rangle$ are algebraically disjoint over $\langle C\rangle$. In the case when $t=0$, $\ddcf$ is a complete stable theory.
\item[5]
\underline{Imaginaries}.
Every completion of $\ddcf$ eliminates imaginaries.
\end{itemize}
It is a little less straightforward to extend the methods of~\cite{paperC} to obtain the appropriate Canonical Base Property and Zilber Dichotomy.
We expect these to hold in a more general setting of fields equipped with commuting free operators, and therefore leave it as the subject of future work.

We conclude by discussing some of the natural reducts of $\ddcf$.

\begin{proposition}
Suppose $(K,\D,\dd)\models \ddcf$.
\begin{enumerate}
\item [(i)] For each associated automorphism $\sigma$, $(K,\D,\sigma)\models \dcfa$.
\item [(ii)] For any $r\leq m$ and $\Delta_1\subseteq\Delta$ a subset of $r$-many derivations, $(K,\D_1,\dd)\models \ddcfr$.
In particular, $(K,\dd)\models \ecdf$.
\item [(iii)]
For $\Delta_1$ as above, and setting $\Delta_2:=\Delta\setminus\Delta_1$, we have that $(K^{\D_2},\D_1,\dd)\models \ddcfr$.
Here $K^{\D_2}$ denotes the field of $\Delta_2$-constants.
\end{enumerate}
\end{proposition}
\begin{proof}
The following are adaptations of the arguments of Proposition 4.12 of \cite{paperC} and Proposition 3.4 of \cite{leonsanchez13}.

(i) By the geometric characterization of the models of $\dcfa$ given in \cite{leonsanchez13}, it suffices to show that if $V$ and $W$ are irreducible $\D$-closed sets such that $W\subset V\times V^{\sigma}$ projects $\D$-dominantly onto each factor, then there is $a\in V$ with $(a,\sigma a)\in W$.
To prove this, consider the pull-back $W'$ of $W$ under $\tau V\to V\times V^{\sigma}$, and apply condition~(III) of Theorem~\ref{geoax} to an irreducible component of $W'$ that projects $\D$-dominantly onto $W$.

(ii)
Checking the axioms for $\ddcfr$, but working with~(III$^\prime$) of~\ref{3prime} rather than~(III) of~\ref{geoax}, we see that it suffices to show that if $\L$ is a characteristic set of a prime $\D_1$-ideal of the $\Delta_1$-polynomial ring over $K$, then $\L$ is also characteristic for a prime $\D$-ideal of $K\{x\}$.
That  this is indeed the case is part of the proof of Kolchin's Irreducibility Theorem \cite[Chapter~0, $\S$6]{kolchin85}.

(iii) We check that $(K^{\D_2},\D_1,\dd)$ is an existentially closed $\DD$-$\D_1$-field. Let $\phi(x)$ be a quantifier free formula in the language of $\DD$-$\D_1$-rings with parameters from $K^{\D_2}$ and with a realisation in some $\DD$-$\D_1$-field  extension $(F,\D_1,\dd)$.
We can expand $(F,\D_1,\dd)$ to a $\cD$-$\D$-field $(F,\D,\dd)$ by interpreting all the derivation in $\Delta_2$ to be trivial on $F$.
So $(F,\D,\dd)$ extends $(K^{\D_2},\D,\dd)$, and itself extends to a model $(L,\D,\dd)\models\ddcf$.
Since $K^{\D_2}$ is a common algebraically closed inversive $\DD$-differential subfield of $(K,\D,\dd)$ and $(L,\D,\dd)$, property~(1) above implies that there is $b\in K$ such that
$$(K,\D,\dd)\models\phi(b)\wedge\bigwedge_{\delta\in\Delta_2}\delta(b)=0$$
Since $\phi$ is quantifier free and refers only to $\D_1$, we have $(K^{\D_2},\D_1,\dd)\models \phi(b)$, as desired.
\end{proof}


\begin{thebibliography}{10}

\bibitem{acfa1}
Z.~Chatzidakis and E.~Hrushovski.
\newblock Model theory of difference fields.
\newblock {\em Transactions of the American Mathematical Society},
  351(8):2997--3071, 1999.

\bibitem{kolchin73}
E.R. Kolchin.
\newblock {\em Differential algebra and algebraic groups}, volume~54 of {\em
  Pure and Applied Mathematics}.
\newblock Academic Press, New York-London, 1973.

\bibitem{kolchin85}
E.R. Kolchin.
\newblock {\em Differential algebraic groups}, volume 114 of {\em Pure and
  Applied Mathematics}.
\newblock Academic Press Inc., 1985.

\bibitem{leonsanchez12}
O.~Le{\'o}n{ }S{\'a}nchez.
\newblock Geometric axioms for differentially closed fields with several
  commuting derivations.
\newblock {\em Journal of Algebra}, 362:107--116, 2012.
\newblock Corrigendum, pages 332--334.

\bibitem{leonsanchez13}
O.~Le\'on{ }S\'anchez.
\newblock On the model companion of partial differential fields with an
  automorphism.
\newblock Preprint, 2013.

\bibitem{paperA}
R.~Moosa and T.~Scanlon.
\newblock Jet and prolongation spaces.
\newblock {\em Journal de l'Institut de Math\'ematiques de Jussieu},
  9(2):391--430, 2010.

\bibitem{paperC}
R.~Moosa and T.~Scanlon.
\newblock Model theory of fields with free operators in characteristic zero.
\newblock Preprint, 2013.

\bibitem{rosenfeld}
Azriel Rosenfeld.
\newblock Specializations in differential algebra.
\newblock {\em Transactions of the American Mathematical Society}, 90:394--407,
  1959.

\bibitem{scanlon2000}
T.~Scanlon.
\newblock A model complete theory of valued ${D}$-fields.
\newblock {\em Journal of Symbolic Logic}, 65(4):1758--1784, 2000.

\bibitem{vdPSinger}
M.~van~der Put and M.~F. Singer.
\newblock {\em Galois Theory of Linear Differential Equations}, volume 328.
\newblock Grundlehren der mathematischen Wissenschaften, Springer, 2003.

\end{thebibliography}

\end{document}